\documentclass[11pt, a4paper,leqno]{amsart}
\usepackage{amsmath,amsthm,amscd,amssymb,amsfonts, amsbsy}
\usepackage{latexsym}
\usepackage{exscale}
\usepackage{txfonts}
\usepackage[colorlinks,citecolor=red,pagebackref,hypertexnames=false]{hyperref}

\usepackage{pgf}
\usepackage{color}

\numberwithin{equation}{section}

\theoremstyle{plain}
\newtheorem{theorem}[equation]{Theorem}
\newtheorem{lemma}[equation]{Lemma}
\newtheorem{corollary}[equation]{Corollary}
\newtheorem{proposition}[equation]{Proposition}

\theoremstyle{definition}
\newtheorem{definition}[equation]{Definition}

\theoremstyle{remark}
\newtheorem{remark}[equation]{Remark}

\newcommand{\dv}{\operatorname{div}}

\newcommand{\supp}{\operatorname{supp}}
\newcommand{\dist}{\operatorname{dist}}

\newcommand{\re}{\mathbb{R}}
\newcommand{\rn}{\mathbb{R}^n}

\newcommand{\rk}{\mathbb{R}^k}

\newcommand{\ree}{\mathbb{R}^{n+1}}

\newcommand{\eps}{\varepsilon}

\newcommand{\vp}{\varphi}

\newcommand{\s}{\mathcal{S}}

\newcommand{\m}{\mathcal{M}}
\newcommand{\W}{\mathcal{W}}

\newcommand{\U}{\Upsilon}
\newcommand{\dsy}{\Delta^*_Y}

\newcommand{\hm}{\omega}

\newcommand{\pom}{\partial\Omega}

\DeclareMathOperator{\diam}{diam}

\DeclareMathOperator{\interior}{int}

\begin{document}

\title[BMO solvability]{BMO solvability and absolute continuity of harmonic measure}

\author{Steve Hofmann}

\address{Steve Hofmann\\
Department of Mathematics\\
University of Missouri\\
Columbia, MO 65211, USA} \email{hofmanns@missouri.edu}

\author{Phi Le}

\address{Phi Le
\\
Department of Mathematics
\\
University of Missouri
\\
Columbia, MO 65211, USA} \email{llc33@mail.missouri.edu}

\thanks{The authors were supported by NSF}

\date{\today}


\begin{abstract} 
We show that for a uniformly elliptic divergence form operator $L$, defined in an open set
$\Omega$ with Ahlfors-David regular boundary, BMO-solvability  implies scale invariant quantitative
absolute continuity (the weak-$A_\infty$ property)
of elliptic-harmonic measure with respect to surface measure on $\pom$.  
We do not impose any connectivity hypothesis, qualitative or quantitative; in particular,
we do not assume the Harnack Chain condition, even within individual connected components of $\Omega$.
In this generality, our results are new even for the Laplacian.
Moreover, 
we obtain a converse, under the additional
assumption that $\Omega$ satisfies an interior Corkscrew condition, in the special case that $L$ is the Laplacian.
\end{abstract}

\maketitle

{\small
\tableofcontents}

\section{Introduction}\label{s1}  The connection between solvability of the Dirichlet problem
with $L^p$ data, 
and scale-invariant absolute continuity properties of harmonic measure
(specifically, that harmonic measure belongs to the Muckenhoupt weight class
$A_\infty$ with respect to surface measure on the boundary), is well documented,
see the monograph of Kenig \cite{Ke}, and the references cited there.  
Specifically, one obtains that the Dirichlet problem is solvable with data
in $L^p(\pom)$ for some $1<p<\infty$, if and only if harmonic measure $\hm$ with some fixed pole is absolutely
continuous with respect to surface measure $\sigma$ on the boundary, and the Poisson
kernel $d\hm/d\sigma$ satisfies a reverse H\"older condition with exponent $p'=p/(p-1)$.
The most general class of domains for which such results had previously been known
to hold is that of the so-called
``1-sided Chord-arc domains" (see Definition \ref{def1.ca} below).  

The connection between solvability of the Dirichlet problem and
scale invariant absolute continuity of harmonic measure was sharpened significantly
in work of Dindos, Kenig and Pipher \cite{DKP}, who showed that harmonic measure satisfies
an $A_\infty$ condition
with respect to surface measure, if and only if a natural
Carleson measure/BMO estimate
(to be described in more detail momentarily)
 holds for solutions of the Dirichlet problem with continuous data.
Their proof was nominally carried out in the setting of a Lipschitz domain,
but more generally, their arguments apply, essentially verbatim, to Chord-arc domains.
The results of \cite{DKP} were recently extended to the setting of a 1-sided Chord-arc domain
by Zihui Zhao \cite{Z}.

More precisely, consider a divergence form elliptic operator
\begin{equation}
\label{eq1.1}
L:=-\dv A(X)\nabla,
\end{equation}
defined in an open set $\Omega\subset\mathbb{R}^{n+1}$, 
where $A$ is $(n+1)\times(n+1)$, real, 
$L^\infty$, 
and satisfies the 
uniform ellipticity condition
\begin{equation}
\label{eq1.1*} \lambda|\xi|^{2}\leq\,\langle A(X)\xi,\xi\rangle
:= \sum_{i,j=1}^{n+1}A_{ij}(X)\xi_{j} \xi_{i}, \quad
  \Vert A\Vert_{L^{\infty}(\mathbb{R}^{n})}\leq\lambda^{-1},
\end{equation}
 for some $\lambda>0$, and for all $\xi\in\ree$, $X\in \Omega$.  

 
Given an open set $\Omega\subset\mathbb{R}^{n+1}$ whose boundary is everywhere regular
in the sense of Weiner,  and a divergence form operator $L$ as above, 
we shall say that the Dirichlet problem is {\it BMO-solvable}\footnote{It might be more 
accurate to refer to this property as ``VMO-solvability", but
BMO-solvability seems to be the established terminology in the literature.
Under less austere circumstances, e.g., in a Lipschitz or (more generally) a Chord-arc domain,
or even in the setting of our Theorem \ref{t2}, where we impose an interior Corkscrew condition,
it can be seen that the two notions are ultimately equivalent (see \cite{DKP} for a discussion of this point),
but in the more general setting of our Theorem \ref{tmain} this matter is not settled.} 
for $L$ in $\Omega$ if 
for all continuous $f$ with compact support on $\pom$, the solution $u$ of the classical 
Dirichlet problem with data $f$ satisfies the Carleson measure estimate
\begin{equation}\label{eq1.2}
\sup_{x\in \pom, \,0<r<r_0} \, \frac{1}{\sigma\big(\Delta(x,r)\big)}\iint_{\Omega\cap B(x,r)} |\nabla u(Y)|^2\, 
\delta(Y)\, dY \leq C \|f\|^2_{BMO(\pom)}\,.
\end{equation}
Here, $r_0:= 10 \diam(\pom)$,
$\sigma$ is surface measure on $\pom$, 
$\delta(Y):=\dist(Y,\pom)$, 
and as usual $B(x,r)$ and $\Delta(x,r):= B(x,r)\cap\pom$ denote, respectively, 
the Euclidean ball in $\ree$, and the surface ball on $\pom$, with center $x$ and radius $r$.

For $X\in \Omega$,
 we let
 $\hm_L^X$  denote elliptic-harmonic measure for $L$ with pole at $X$, and if the dependence on
 $L$ is clear in context, we shall simply write $\hm^X$.

The main result of this paper is the following.  All terminology used in the statement of the theorem and not 
discussed
already, will be defined precisely in the sequel.

\begin{theorem}\label{tmain}  Suppose that $\Omega\subset \ree, n\geq 2$, is an open set, not necessarily
connected, with Ahlfors-David Regular boundary.  Let $L$ be a divergence form elliptic operator defined on
$\Omega$. If the Dirichlet problem for $L$ is BMO-solvable in $\Omega$, then harmonic measure belongs
to weak-$A_\infty$ in the following sense:
for every ball $B=B(x,r)$, with $x\in \pom$, and $0<r<\diam(\pom)$, and for all $Y\in \Omega\setminus 4B$,
harmonic measure $\hm_L^Y\in$ weak-$A_\infty(\Delta)$, where $\Delta := B\cap\pom$, and 
where the parameters in the weak-$A_\infty$ condition are uniform in $\Delta$, 
and in $Y\in \Omega\setminus4B$.  
\end{theorem}

 As mentioned above, this result was 
 established in \cite{DKP}, and in \cite{Z},  under the more restrictive assumption that
 $\Omega$ is Chord-arc, or 1-sided Chord-arc, respectively.  The arguments of \cite{DKP} and \cite{Z}
 rely both explictly and implicitly on quantitative connectivity of the domain, more precisely,
on  the Harnack Chain condition (see Definition \ref{def1.hc} below).
The new contribution of the present paper is to dispense with all connectivity assumptions,
both qualitative and quantitative.
In particular, we do not assume the Harnack Chain condition, even within
individual connected components of $\Omega$.  
 In this generality, our results are new even
for the Laplacian.   

We observe that we draw a slightly weaker conclusion than that of \cite{DKP}
(or \cite{Z}), namely, weak-$A_\infty$, as opposed to $A_\infty$, but this is the best that can be hoped for
in the absence of connectivity:  indeed, clearly, the doubling property of harmonic measure
may fail without connectivity.  Moreover, even in a connected domain 
enjoying an
``interior big pieces of Lipschitz domains" condition, and having an ADR boundary
(and thus, for which harmonic measure belongs to weak-$A_\infty$, by the main result of \cite{BL}), 
 the doubling property may fail
in the absence of Harnack Chains;  see \cite[Section 4]{BL} for a counter-example.

In the particular case that $L$ is the Laplacian, we also obtain the following. 

\begin{corollary}\label{cor1.5}
Let $\Omega\subset \ree, n\geq 2$, be an open set, not necessarily
connected, with Ahlfors-David Regular boundary, and in addition, suppose that 
$\Omega$ satisfies an interior Corkscrew condition (Definition \ref{def1.cork}),
and that the Dirichlet problem is BMO-solvable for Laplace's equation in $\Omega$.
Then $\pom$  is uniformly rectifiable (Definition \ref{defur}).
\end{corollary}

The proof of the corollary is almost immediate:  by Theorem \ref{tmain}, harmonic measure
belongs to weak-$A_\infty$ (even without the Corkscrew condition), 
so by the result of \cite{HM-IV}\footnote{See also \cite{HLMN} and \cite{MT} for more general versions
of the result of \cite{HM-IV}.}, 
in the presence of the interior Corkscrew condition, 
$\pom$ is uniformly rectifiable.

We remark that the Corkscrew hypothesis is fairly mild, in the sense that
if $\Omega=\ree\setminus E$ is the complement of an ADR set, then the 
Corkscrew condition holds automatically, by a simple pigeon-holing argument.
We also remark that in the absence of the Corkscrew condition, the result of \cite{HM-IV} may fail;
a counter-example will appear in forthcoming work of the first author and J. M. Martell.

We also obtain a partial converse to Theorem \ref{tmain}.

\begin{theorem}\label{t2}  Let  $\Omega\subset \ree, n\geq 2$ be an open set, not necessarily
connected, with Ahlfors-David Regular (ADR) boundary.  
Let $L$ be a divergence form elliptic operator defined on
$\Omega$, and suppose that elliptic-harmonic measure for $L$ belongs to weak-$A_\infty$ in the sense of
the conclusion of Theorem \ref{tmain}.  Then the Dirichlet problem for $L$
is $L^p$-solvable\footnote{We shall say more precisely what this means, in the sequel; see the 
statement of Proposition \ref{prop4.2}.} in $\Omega$,
for $p<\infty$ sufficiently large.   In the special case that $L$ is the Laplacian,  the Dirichlet problem
  is BMO-solvable, provided also that $\Omega$ satisfies an interior Corkscrew condition.
\end{theorem}

As noted above, our main new 
contribution is Theorem \ref{tmain}, which establishes the direction
BMO-solvability implies $\hm \in$ weak-$A_\infty$;  it is in
that direction that the lack of connectivity is most problematic.
By contrast, our proof  of the opposite implication (i.e., Theorem \ref{t2}) 
is a fairly routine adaptation of the corresponding
arguments of \cite{DKP} and of \cite{FN}.  On the other hand, let us point out that 
in Theorem \ref{t2}, we have imposed an extra assumption, namely the Corkscrew condition.
At present, we do not know whether the latter hypothesis is necessary to obtain the conclusion of Theorem
\ref{t2} (although as remarked above, in its absence uniform rectifiability of $\pom$ may fail), 
nor do we know whether the conclusion of BMO solvability extends to the case of a
general divergence form elliptic operator $L$.

To provide some further context for our results here, let us mention that recently,
Kenig, Kirchheim, Pipher and Toro have shown in \cite{KKiPT} that for a Lipschitz domain 
$\Omega$,  a weaker Carleson measure estimate,
namely, a version of \eqref{eq1.2} in which the BMO norm of the boundary data is replaced by
$\|u\|_{L^\infty(\Omega)}$, still suffices to establish that $\hm_L$ satisfies an $A_\infty$ condition
with respect to surface measure on $\pom$.  Moreover, the argument of \cite{KKiPT} carries over with
minor changes to the more general setting of a uniform (i.e., 1-sided NTA) domain with
Ahlfors-David regular boundary \cite{HMT}.   However, in contrast to our Theorem \ref{tmain}, to deduce
absolute continuity of harmonic measure under the weaker $L^\infty$ Carleson measure condition seems 
necessarily to
require some sort of connectivity (such as the Harnack Chain condition enjoyed by uniform domains).
Indeed, specializing to the case that $L$ is the Laplacian,
an example of Bishop and Jones \cite{BiJo} shows that  harmonic measure $\hm$ need {\it not}
be absolutely continuous with respect to surface measure, even for domains with uniformly rectifiable boundaries,
whereas the first named author of this paper, along with J. M. Martell and S. Mayboroda,
have shown in \cite{HMM} that uniform rectifiability of $\pom$ alone suffices to deduce the $L^\infty$
version of \eqref{eq1.2} in the harmonic case (and indeed, for solutions of certain other
elliptic equations as well).

The paper is organized as follows.  In the remainder of this section, 
we present some basic notations and definitions.  In Section \ref{s2}, we recall some
known results from the theory of elliptic PDE.  In Sections \ref{s4} and \ref{s3}, we give the proofs of
Theorems \ref{tmain} and \ref{t2}, respectively.

\subsection{Notation and Definitions}\label{ss1.1}

\begin{list}{$\bullet$}{\leftmargin=0.4cm  \itemsep=0.2cm}

\item We use the letters $c,C$ to denote harmless positive constants, not necessarily
the same at each occurrence, which depend only on dimension and the
constants appearing in the hypotheses of the theorems (which we refer to as the
``allowable parameters'').  We shall also
sometimes write $a\lesssim b$ and $a \approx b$ to mean, respectively,
that $a \leq C b$ and $0< c \leq a/b\leq C$, where the constants $c$ and $C$ are as above, unless
explicitly noted to the contrary.  

\item Given a closed set $E \subset \ree$, we shall
use lower case letters $x,y,z$, etc., to denote points on $E$, and capital letters
$X,Y,Z$, etc., to denote generic points in $\ree$ (especially those in $\ree\setminus E$).

\item The open $(n+1)$-dimensional Euclidean ball of radius $r$ will be denoted
$B(x,r)$ when the center $x$ lies on $E$, or $B(X,r)$ when the center
$X \in \ree\setminus E$.  A ``surface ball'' is denoted
$\Delta(x,r):= B(x,r) \cap\partial\Omega.$

\item Given a Euclidean ball $B$ or surface ball $\Delta$, its radius will be denoted
$r_B$ or $r_\Delta$, respectively.

\item Given a Euclidean or surface ball $B= B(X,r)$ or $\Delta = \Delta(x,r)$, its concentric
dilate by a factor of $\kappa >0$ will be denoted
$\kappa B := B(X,\kappa r)$ or $\kappa \Delta := \Delta(x,\kappa r).$

\item Given a (fixed) closed set $E \subset \ree$, for $X \in \ree$, we set $\delta(X):= \dist(X,E)$.

\item We let $H^n$ denote $n$-dimensional Hausdorff measure, and let
$\sigma := H^n\lfloor_{E}$ denote the ``surface measure'' on a closed set $E$
of co-dimension 1.

\item For a Borel set $A\subset \ree$, we let $1_A$ denote the usual
indicator function of $A$, i.e. $1_A(x) = 1$ if $x\in A$, and $1_A(x)= 0$ if $x\notin A$.

\item For a Borel set $A\subset \ree$,  we let $\interior(A)$ denote the interior of $A$.


\item Given a Borel measure $\mu$, and a Borel set $A$, with positive and finite $\mu$ measure, we
set $\fint_A f d\mu := \mu(A)^{-1} \int_A f d\mu$.

\item We shall use the letter $I$ (and sometimes $J$)
to denote a closed $(n+1)$-dimensional Euclidean dyadic cube with sides
parallel to the co-ordinate axes, and we let $\ell(I)$ denote the side length of $I$.
If $\ell(I) =2^{-k}$, then we set $k_I:= k$.


\end{list}

\begin{definition}\label{defadr} ({\bf  ADR})  (aka {\it Ahlfors-David regular}).
We say that a  set $E \subset \ree$, of Hausdorff dimension $n$, is ADR
if it is closed, and if there is some uniform constant $C$ such that
\begin{equation} \label{eq1.ADR}
\frac1C\, r^n \leq \sigma\big(\Delta(x,r)\big)
\leq C\, r^n,\quad\forall r\in(0,\diam (E)),\ x \in E,
\end{equation}
where $\diam(E)$ may be infinite.
\end{definition}

\begin{definition}\label{defur} ({\bf UR}) (aka {\it uniformly rectifiable}).
An $n$-dimensional ADR (hence closed) set $E\subset \ree$
is UR if and only if it contains ``Big Pieces of
Lipschitz Images" of $\rn$ (``BPLI").   This means that there are positive constants $\theta$ and
$M_0$, such that for each
$x\in E$ and each $r\in (0,\diam (E))$, there is a
Lipschitz mapping $\rho= \rho_{x,r}: \rn\to \ree$, with Lipschitz constant
no larger than $M_0$,
such that 
$$
H^n\Big(E\cap B(x,r)\cap  \rho\left(\{z\in\rn:|z|<r\}\right)\Big)\,\geq\,\theta\, r^n\,.
$$
\end{definition}

We recall that $n$-dimensional rectifiable sets are characterized by the
property that they can be
covered, up to a set of
$H^n$ measure 0, by a countable union of Lipschitz images of $\rn$;
we observe that BPLI  is a quantitative version
of this fact.

We remark
that, at least among the class of ADR sets, the UR sets
are precisely those for which all ``sufficiently nice" singular integrals
are $L^2$-bounded  \cite{DS1}.    In fact, for $n$-dimensional ADR sets
in $\ree$, the $L^2$ boundedness of certain special singular integral operators
(the ``Riesz Transforms"), suffices to characterize uniform rectifiability (see \cite{MMV} for the case $n=1$, and 
\cite{NToV} in general). 
We further remark that
there exist sets that are ADR (and that even form the boundary of a domain satisfying 
interior Corkscrew and Harnack Chain conditions),
but that are totally non-rectifiable (e.g., see the construction of Garnett's ``4-corners Cantor set"
in \cite[Chapter1]{DS2}).  Finally, we mention that there are numerous other characterizations of UR sets
(many of which remain valid in higher co-dimensions);   see \cite{DS1,DS2}.

\begin{definition} ({\bf Corkscrew condition}).  \label{def1.cork}
Following
\cite{JK}, we say that an open set $\Omega\subset \ree$
satisfies the {\it Corkscrew condition} (more precisely,
the {\it interior} Corkscrew condition) if for some uniform constant $c>0$ and
for every surface ball $\Delta:=\Delta(x,r),$ with $x\in \partial\Omega$ and
$0<r<\diam(\partial\Omega)$, there is a ball
$B(X_\Delta,cr)\subset B(x,r)\cap\Omega$.  The point $X_\Delta\subset \Omega$ is called
a ``Corkscrew point'' relative to $\Delta.$  
\end{definition}

\begin{definition}({\bf Harnack Chain condition}).  \label{def1.hc} Again following \cite{JK}, we say
that $\Omega$ satisfies the {\it Harnack Chain condition} if there is a uniform constant $C$ such that
for every $\rho >0,\, \Lambda\geq 1$, and every pair of points
$X,X' \in \Omega$ with $\delta(X),\,\delta(X') \geq\rho$ and $|X-X'|<\Lambda\,\rho$, there is a chain of
open balls
$B_1,\dots,B_N \subset \Omega$, $N\leq C(\Lambda)$,
with $X\in B_1,\, X'\in B_N,$ $B_k\cap B_{k+1}\neq \emptyset$
and $C^{-1}\diam (B_k) \leq \dist (B_k,\partial\Omega)\leq C\diam (B_k).$  The chain of balls is called
a ``Harnack Chain''.
\end{definition}

\begin{definition}({\bf NTA and uniform domains}). \label{def1.nta} Again following \cite{JK}, we say that a
domain $\Omega\subset \ree$ is NTA (``Non-tangentially accessible") if it satisfies the
Harnack Chain condition, and if both $\Omega$ and
$\Omega_{\rm ext}:= \ree\setminus \overline{\Omega}$ satisfy the Corkscrew condition.
If $\Omega$ merely satisfies the Harnack Chain condition and the interior 
(but not exterior) Corkscrew condition,
then it is said to be a {\it uniform} (aka {\it 1-sided NTA}) domain.
\end{definition}

\begin{definition}({\bf Chord-arc and 1-sided Chord-arc}). \label{def1.ca} A
domain $\Omega\subset \ree$ is {\it Chord-arc} if it is an NTA domain with an ADR boundary;
it is {\it 1-sided Chord-arc} if it is a uniform (i.e., 1-sided NTA) domain with ADR boundary.
\end{definition}

\begin{definition}\label{defAinfty}
({\bf $A_\infty$}, weak-$A_\infty$, and weak-$RH_q$). 
Given an ADR set $E\subset\ree$, 
and a surface ball
$\Delta_0:= B_0 \cap E$,
we say that a Borel measure $\mu$ defined on $E$ belongs to
$A_\infty(\Delta_0)$ if there are positive constants $C$ and $\theta$
such that for each surface ball $\Delta = B\cap E$, with $B\subseteq B_0$,
we have
\begin{equation}\label{eq1.ainfty}
\mu (F) \leq C \left(\frac{\sigma(F)}{\sigma(\Delta)}\right)^\theta\,\mu (\Delta)\,,
\qquad \mbox{for every Borel set } F\subset \Delta\,.
\end{equation}
Similarly, we say that $\mu \in$ weak-$A_\infty(\Delta_0)$ if 
for each surface ball $\Delta = B\cap E$, with $2B\subseteq B_0$,
\begin{equation}\label{eq1.wainfty}
\mu (F) \leq C \left(\frac{\sigma(F)}{\sigma(\Delta)}\right)^\theta\,\mu (2\Delta)\,,
\qquad \mbox{for every Borel set } F\subset \Delta\,.
\end{equation}
We recall that, as is well known, the condition $\mu \in$ weak-$A_\infty(\Delta_0)$
is equivalent to the property that $\mu \ll \sigma$ in $\Delta_0$, and that for some $q>1$, the
Radon-Nikodym derivative $k:= d\mu/d\sigma$ satisfies
the weak reverse H\"older estimate
\begin{equation}\label{eq1.wRH}
\left(\fint_\Delta k^q d\sigma \right)^{1/q} \,\lesssim\, \fint_{2\Delta} k \,d\sigma\,
\approx\,  \frac{\mu(2\Delta)}{\sigma(\Delta)}\,,
\quad \forall\, \Delta = B\cap E,\,\, {\rm with} \,\, 2B\subseteq B_0\,.
\end{equation}
We shall refer to the inequality in \eqref{eq1.wRH} as
an  ``$RH_q$" estimate, and we shall say that $k\in RH_q(\Delta_0)$ if $k$ satisfies \eqref{eq1.wRH}.
\end{definition}

\section{Preliminaries}\label{s2}
In this section, we record some known estimates for elliptic harmonic measure $\hm_L$
associated to
a divergence form operator $L$ as in \eqref{eq1.1} and \eqref{eq1.1*}, and for solutions of
the equation $Lu=0$, in an open set $\Omega\subset \ree$ 
with an ADR boundary.  In the sequel, we shall always assume that the
ambient dimension $n+1\geq 3$.  We recall that, as a consequence
of the ADR property, 
every point on $\pom$ is regular in the sense of Wiener (see, e.g., \cite[Remark 3.26, Lemma 3.27]{HLMN}).

\begin{lemma}[Bourgain \cite{B}]\label{Bourgainhm}  Let $\Omega\subset \ree$ be an open set,
and suppose that
$\partial \Omega$ is $n$-dimensional ADR.  Then there are uniform constants $c\in(0,1)$
and $C\in (1,\infty)$, depending only on $n$, ADR, and the ellipticity parameter $\lambda$,
such that for every $x \in \partial\Omega$, and every $r\in (0,\diam(\partial\Omega))$,
if $Y \in \Omega \cap B(x,cr),$ then
\begin{equation}\label{eq2.Bourgain1}
\omega_L^{Y} (\Delta(x,r)) \geq 1/C>0 \;.
\end{equation}
\end{lemma}
We refer the reader to \cite[Lemma 1]{B} for the proof in the case that $L$ is the Laplacian,
but the proof is the same for a general uniformly elliptic divergence form operator.

We note for future reference that
in particular,  if $\hat{x}\in \pom$
satisfies
$|X-\hat{x}|=\delta(X)$, and
$\Delta_X:= \pom\cap B\big(\hat{x}, 10\delta(X)\big)$, 
then for a slightly different uniform constant $C>0$,
\begin{equation}\label{eq2.Bourgain2}
\omega_L^{X} (\Delta_X) \geq 1/C \;.
\end{equation}
Indeed, the latter bound follows immediately from \eqref{eq2.Bourgain1},
and the fact that we can form a Harnack Chain connecting
$X$ to a point $Y$ that lies on the line segment from $X$ to $\hat{x}$, and satisfies $|Y-\hat{x}|= c\delta(X)$.

As a consequence of Lemma \ref{Bourgainhm}, we have the following.
\begin{corollary}\label{cor2.4} Let $\Omega\subset \ree$ be an open set,
and suppose that
$\partial \Omega$ is $n$-dimensional ADR.  For $x\in \pom$, and $0<r<\diam \pom$,
let $u$ be a non-negative solution of $Lu=0$ in $\Omega\cap B(x,2r)$, 
which vanishes continuously on 
$\Delta(x,2r) = B(x,2r)\cap\pom$.  Then for some $\alpha >0$,
\begin{equation}\label{eq2.5}
u(Y) \leq C \left(\frac{\delta(Y)}{r}\right)^\alpha \frac1{|B(x,2r)|}\,\int\!\!\!\int_{B(x,2r)\cap\Omega} u\,,\qquad \forall\, Y\in B(x,r)\cap\Omega\,,
\end{equation}
where the constants $C$ and $\alpha$ depend 
only on $n$, ADR and $\lambda$.
\end{corollary}

\section{Proof of Theorem \ref{tmain}:  BMO-solvability implies $\hm\in$ weak-$A_\infty$}\label{s4}
The basic outline of the proof follows that of \cite{DKP}, but the lack of Harnack Chains
requires in addition some slightly delicate geometric arguments inspired in part by the work of Bennewitz and Lewis
\cite{BL}.

We begin by recalling the following deep fact, established in \cite{BL}.  Given
a point $X\in \Omega$,  
let $\hat{x}\in \pom$ be a ``touching point" for the ball $B(X,\delta(X))$, i.e., $|X-\hat{x}|=\delta(X)$.
Set 
\begin{equation}\label{eq3.1}
\Delta_X:= \Delta\big(\hat{x}, 10 \delta(X)\big)\,.
\end{equation}
\begin{lemma}
\label{BLlemma} Let $\pom$ be ADR, and
suppose that there are constants $c_0,\eta \in(0, 1)$, such that
for each $X\in \Omega$,  with $\delta(X)<\diam(\pom)$,
and for every Borel set $F\subset \Delta_X$,
\begin{equation}\label{eq3.3}
\sigma(F)\geq (1-\eta) \sigma(\Delta_X)  \, \, \implies \,\,  \hm^X(F) \geq c_0\,.
\end{equation}
Then $\hm^Y\in$ weak-$A_\infty(\Delta)$, where $\Delta=B\cap\pom$,
for every ball $B=B(x,r)$, with $x\in \pom$ and $0<r<\diam(\pom)$, and for
all $Y\in \Omega\setminus 4B$.  Moreover, the parameters in the weak-$A_\infty$
condition depend only on $n$, ADR, $\eta$, $c_0$, and the ellipticity parameter $\lambda$
of the divergence form operator $L$.
\end{lemma}

\begin{remark} Lemma \ref{BLlemma} is not stated explicitly in this form in \cite{BL},
but may be gleaned readily from the combination of \cite[Lemma 2.2]{BL} and its proof,
and \cite[Lemma 3.1]{BL}.   We mention also that the paper \cite{BL} treats explictly
only the case that $L$ is the Laplacian, but the proof of  \cite[Lemma 2.2]{BL}
carries over verbatim to the case of a general uniformly elliptic divergence form operator with real coefficients,
while \cite[Lemma 3.1]{BL} is a purely real variable result.
\end{remark}

Given the BMO-solvability estimate \eqref{eq1.2}, it suffices to verify the hypotheses of Lemma
\ref{BLlemma}, with $\eta$ and $c_0$ depending only on $n$, ADR, $\lambda$, and the constant $C$ in
\eqref{eq1.2}.  To this end, let $X\in \Omega$,   $\delta(X)<\diam(\pom)$,
and for notational convenience,
set
$$r:= \delta(X)\,.$$
We choose $\hat{x}\in \pom$ so that
$|X-\hat{x}|=r$, and let $a\in (0, \pi/10000)$ be a sufficiently small number to be chosen
depending only on $n$ and ADR.    We then define $\Delta_X$ as in \eqref{eq3.1}, and set 
\begin{equation}\label{eq4.5}
B_X:= B(\hat{x}, 10 r)\,,\quad B'_X:= B(\hat{x}, a r)\,,\quad \Delta_X':= \Delta\big(\hat{x},ar)\,.
\end{equation}
We make the following pair of claims.

\smallskip

\noindent {\bf Claim 1}.   For $a$ small enough, depending only on $n$ and ADR,
there is a constant $\beta>0$ depending only on  $n$, $a$, ADR and $\lambda$, and
a ball $B_1:= B(x_1, ar)\subset B_X$, with $x_1\in \pom$,
such that  $\dist(B_X',B_1) \geq 5a r$, and
\begin{equation}\label{eq3.4a}
\hm_L^X(\Delta_1) \,\geq \,\beta\, \hm_L^X(\Delta_X)\,,
\end{equation}
where $\Delta_1:= B_1\cap\pom$.


\smallskip

\noindent {\bf Claim 2}.  Suppose that $u$ is a
non-negative solution of $Lu=0$ in $\Omega$, vanishing 
continuously on $2\Delta_X'$, with $\|u\|_{L^\infty(\Omega)}\leq 1$.  Then for every $\eps >0$, 
\begin{equation}\label{eq3.4}
u(X) \leq C_\eps \left(\frac{1}{\sigma\big(\Delta_X\big)}
\iint_{B_X\cap\Omega}|\nabla u(Y)|^2\delta(Y)\, dY\right)^{1/2}\, + \,C\eps^\alpha\,,
\end{equation}
where $\alpha >0$ is the H\"older exponent in Corollary \ref{cor2.4}.

\smallskip

Momentarily taking these two claims for granted, we now follow the argument in \cite{DKP}, with some minor
modifications, in order to establish the hypotheses of Lemma \ref{BLlemma}. 
Let $B_1$ and $\Delta_1$ be as in Claim 1.
Let $F\subset \Delta_X$ be a Borel set satisfying the first inequality in \eqref{eq3.3}, for some
small $\eta>0$.  If we choose $\eta$ small enough,  depending only on $n$, ADR, 
and the constant $a$ in the definition of $B_X'$, then 
$$\sigma(F_1)\, \geq \,\big(1-\sqrt{\eta}\,\big)\, \sigma(\Delta_1)\,,$$
where $F_1:= F\cap \Delta_1$. 
Set $A_1:= \Delta_1\setminus F_1$, and define
$$ f:= \max\big(0, 1+\gamma \log \m(1_{A_1})\big)\,,$$
where $\gamma$ is a small number to be chosen, and $\m$ is the usual
Hardy-Littlewood maximal
operator on $\pom$.  
Note that
\begin{equation}\label{eq3.5}
0\leq f\leq 1\,,\qquad \|f\|_{BMO(\pom)} \leq C\gamma\,,\qquad  1_{A_1}\leq f\,.
\end{equation}
Note also that if $z\in \pom \setminus 2B_1$, then
$$\m(1_{A_1})(z) \,\lesssim\,\frac{\sigma(A_1)}{\sigma(\Delta_1)}\, \lesssim\, \sqrt{\eta}\,,$$ 
where the implicit constants
depend only on $n$ and ADR.  
Thus, if $\eta$ is chosen small enough depending on $\gamma$, then
$1+\gamma \log \m(1_{A_1})$ will be negative, hence $f\equiv 0$, on $\pom \setminus 2B_1$.

In order to work with continuous data, we shall require the following.
\begin{lemma}\label{l3.9}  There exists a collection of continuous functions $\{f_s\}_{0<s<ar/1000}$, 
defined on $\pom$, with the following properties.
\begin{enumerate}
\item $0\leq f_s\leq 1$, for each $s$.
\smallskip
\item $\supp(f_s)\subset 3B_1 \cap\pom$.
\smallskip
\item $1_{A_1}(z) \leq \liminf_{s\to 0} f_s(z)$, for  $\hm^X$-a.e. $z\in \pom$.
\smallskip
\item  $\sup_s\|f_s\|_{BMO(\pom)} \leq C \|f\|_{BMO(\pom)} \lesssim \gamma$, where $C=C(n,ADR)$.
\end{enumerate}
\end{lemma}
The proof is based on a standard mollification of the function
$f$ constructed above.  We defer the routine proof to the end of this section.

Let $u_s$ be the solution of the Dirichlet problem for the equation
$Lu_s=0$ in $\Omega$, with data $f_s$.
Then, for a small $\eps>0$ to be chosen momentarily, by Lemma \ref{l3.9}, Fatou's lemma,
and Claim 2, we have
\begin{equation}\label{eq3.10}
\hm_L^X(A_1) \leq\, \int_{\pom} \liminf_{s\to0} f_s\, d\hm^X\, \leq \,\liminf_{s\to0} u_s(X)\,
\leq \, C_\eps \gamma  \,+\, C\eps^\alpha\,,
\end{equation}
where in the last step we have used \eqref{eq3.4}, \eqref{eq1.2}, and 
Lemma \ref{l3.9}-(4).  Combining \eqref{eq3.10} with \eqref{eq2.Bourgain2}, we find that
\begin{equation}\label{eq4.8}
\hm_L^X(A_1) \,\leq \, \big(C_\eps \gamma  \,+\, C\eps^\alpha\big)\, \hm_L^X(\Delta_X)\,.
\end{equation}

Next, we set $A:= \Delta_X\setminus F$, and observe that by definition of $A$ and $A_1$,
along with Claim 1, and \eqref{eq4.8},
$$\hm^X_L(A) \,\leq \, \hm^X_L(\Delta_X \setminus \Delta_1) \, +\, \hm^X_L(A_1)
\,\leq \,  \big(1-\beta \,+\,C_\eps \gamma  \,+\, C\eps^\alpha\big)\, \hm_L^X(\Delta_X)\,.$$
We now choose first $\eps>0$, and then $\gamma>0$, so that 
$C_\eps \gamma  \,+\, C\eps^\alpha <\beta/2$, to obtain that
$$\hm^X_L(F) \,\geq\, \frac{\beta}{2} \,\hm_L^X(\Delta_X) \, \geq \,c\beta\,,$$
where in the last step we have used  \eqref{eq2.Bourgain2}.

It now remains only to establish the two claims, and to prove Lemma \ref{l3.9}.

\begin{proof}[Proof of Claim 1]
By translation and rotation, we may suppose without loss of generality that
$\hat{x} =0$, and that the line segment joining $\hat{x}$ to $X$ is purely vertical,
thus, $X = re_{n+1}$, where as usual $e_{n+1} := (0,...,0,1)$.  
 Let $\Gamma, \Gamma',\Gamma''$ denote, respectively, 
the open inverted vertical cones with vertex at $X$ having angular apertures $200a$,
$100a$, and $20a$, respectively (recall that $a < \pi/10000$).  Then $B_X'\subset \Gamma''$ 
(where $B_X'$ is defined in \eqref{eq4.5}).  Recalling that $r=\delta(X)$, we
let $B_0:= B(X,r)$ denote the open ``touching ball", so that $B_0\cap \pom =\emptyset$,
and define a closed annular region $R_0 := \overline{5B_0} \setminus B_0$.
We now consider two cases:

\smallskip

\noindent {\bf Case 1}.  $ \pom \cap (R_0\setminus \Gamma)$ is non-empty.  In this case,
we let  $x_1$ be the point 
in  $ \pom \cap (R_0\setminus \Gamma)$ that is closest to $X$ (if there is more than one
such point, we just pick one).  Then by construction $r\leq  |X-x_1|\leq 5r$, and the ball 
$B_1=B(x_1,ar)$ misses $\Gamma'$, hence
$\dist(B_1, B_X') \geq \dist(B_1,\Gamma'') >5ar$.  Moreover, since $x_1$ is the closest point to
$X$, setting
$\rho:= |X-x_1|$, we have that 
$\Omega'
\cap\pom =\emptyset$, where
$$\Omega':= \big(B(X,\rho)\setminus \overline{\Gamma}\big) \cup B_0\,.$$
Consequently, we may construct a Harnack Chain within the subdomain
$\Omega'\subset \Omega$, 
connecting $X$ to a point $Y\in B(x_1,car)\cap
\Omega'$, with $\delta(Y) \geq cr/2$, where $c$ is the constant in Lemma \ref{Bourgainhm}.
Thus, by Harnack's inequality and Lemma \ref{Bourgainhm}, 
$$\hm_L^X(\Delta_1) \gtrsim \hm_L^Y(\Delta_1) \geq 1/C\,.$$
Since $\hm_L^X(\Delta_X) \leq 1$, we obtain \eqref{eq3.4a}, and thus Claim 1
holds in the present case.

\smallskip

\noindent {\bf Case 2}.  $ \pom \cap (R_0\setminus \Gamma) = \emptyset$.
By ADR, we have that
$$\sigma\big(\Delta(0, 10ar)\big) \leq C (ar)^n\,,\qquad \sigma\big(B(X, 4r)\cap\pom  \big) \geq r^n/C\,.$$
Thus, for $a$ chosen small enough, depending only on $n$ and ADR, we see that 
the set 
$ \pom \cap \big(B(X, 4r) \setminus B(0, 10ar)\big)$ is non-empty.  
Consequently, under the scenario of Case 2,
$$ \pom \cap \Big(\,\overline{B(X, 4r)} \setminus 
B(0, 10ar) 
\Big) \subset \Gamma\,.$$
Define
$$\theta_0 := \min\big\{\theta \in [0,200a): \pom \cap \big(\overline{B(X, 4r)} \setminus B(0, 10ar)\big)
\subset \Gamma_{\theta} \big\}\,,$$
where $\Gamma_\theta$ is the inverted cone
with vertex at $X$ of 
angular
aperture $\theta$ (if $n+1=2$, it may happen that
$\theta_0 =0$, in which case $\pom \cap \big(\overline{B(X, 4r)} \setminus B(0, 10ar)$
is contained in the vertical ray pointing straight downward from 0).  Then by construction,
there is a point 
$$x_1 \in \partial \Gamma_{\theta_0} \cap \pom \cap \Big(\overline{B(X, 4r)} \setminus B(0, 10ar)\Big)$$
(or, as noted above, $x_1$ lies on the downward vertical ray if $n+1=2$ and $\theta_0=0$).
Then $B_1=B(x_1,ar)$ misses $B(0,9ar)$, so that in particular,
$\dist(B_1,B_X') >5ar$.  Moreover, $\Omega'\cap\pom=\emptyset$, where now
$$\Omega':= \Big(\big(B(X,4r)\setminus \overline{\Gamma_{\theta_0}}\big)\cup B_0\Big)
\setminus \overline{B(0, 10ar)}\,
$$
(with the obvious adjustment if $n+1=2$ and $\theta_0=0$).
Thus, as in Case 1, there is a point $Y \in B(x_1,car) \cap\Omega'$, with $\delta(Y)>cr/2$,
which may be joined to $X$ via a Harnack Chain within the subdomain $\Omega'\subset \Omega$,
whence by Harnack's inequality, Lemma \ref{Bourgainhm}, and the fact that 
$\hm_L^X(\Delta_X) \leq 1$, we again obtain \eqref{eq3.4a}.  Claim 1 therefore
holds in all cases.
\end{proof}  

\begin{proof}[Proof of Claim 2]
As in the proof of Claim 1, we may assume by translation and rotation that
$\hat{x}=0$, and that $X=re_{n+1}$, with $r=\delta(X)$.
Let $\Gamma$ denote the {\it upward} open vertical cone with vertex at 0, of angular aperture
$\pi/100$.  We let $S$ denote the spherical cap inside $\Gamma$, i.e.,
$S:= S^n\cap\Gamma$ (recall that our ambient dimension is $n+1$).
Then by Harnack's inequality, letting $\mu$ denote surface measure on the unit sphere, we have
$$u(X) \lesssim \int_{S} u(r\xi) \, d\mu(\xi) = \, 
 \int_{S} \Big(u(r\xi)-u(\eps r \xi)\Big) \, d\mu(\xi)\,+\, O(\eps^\alpha)\, =:\, I +O(\eps^\alpha)\,,$$
 where we have used Corollary \ref{cor2.4} to estimate the ``big-O" term.
 In turn,
 $$|I| = \Big|\int_S \!\int_{\eps r}^r \frac{\partial}{\partial t} \big(u(t\xi)\big)\, dt \, d\mu(\xi)\,\Big|\,
 \leq\, (\eps r)^{-n} \iint_{\Gamma \cap R_{\eps}}|\nabla u(Y) | \, dY\,,$$
 where $R_{\eps}:=  B(0,r) \setminus B(0,\eps r)$, and we have used polar co-ordinates in
 $n+1$ dimensions.  We then have
\begin{multline*}
|I| \lesssim (\eps r)^{-n} r^{(n+1)/2} \left(\iint_{\Gamma \cap R_{\eps}}|\nabla u(Y) |^2 \, dY\right)^{1/2}
\\
\lesssim \, (\eps)^{-n-1/2} r^{-n/2} \left(\iint_{B(0,r)\cap\Omega}|\nabla u(Y) |^2 \, 
\delta(Y)\,dY\right)^{1/2}\,,
\end{multline*}
where  we have used that by construction, 
$\Gamma \cap R_{\eps} \subset B(0,r)\cap\Omega$, with $\delta(Y) \approx |Y| \geq \eps r$ in $\Gamma \cap
R_\eps$.  Estimate \eqref{eq3.4} now follows, by ADR and the definition of $B_X$.
\end{proof}

\begin{proof}[Proof of Lemma \ref{l3.9}]
Let $\zeta \in C^\infty_0(\ree)$, with 
$$\supp(\zeta)\subset B(0,1)\,,\quad \zeta \equiv 1 \,\, {\rm on}\,\, B(0,1/2)\,,\quad 0\leq \zeta\leq 1\,.  $$
Given $s\in (0,ar/1000)$, and $z,y\in \pom$, set   
$$\Lambda_s(z,y):= b(z,s)^{-1} \zeta\big(s^{-1}(z-y)\big)\,,$$
where 
\begin{equation}\label{eq3.11}
b(z,s):= \int_{\pom} \zeta\big(s^{-1}(z-y)\big)\, d\sigma(y)\,\approx s^n\,,
\end{equation}
uniformly in $z\in \pom$, by the ADR property.  Furthermore,
$$\int_{\pom}  \Lambda_s(z,y) \, d\sigma(y) \equiv 1\,,\qquad \forall\, z\in \pom\,.$$
We now define
$$f_s(z) := \int_{\pom}  \Lambda_s(z,y) \, f(y)\,d\sigma(y) \,,$$
so that $f_s$ is continuous, by construction.  Let us now verify (1)-(4) of
Lemma \ref{l3.9}.   We obtain (1) immediately, by \eqref{eq3.5}, and the properties of $\Lambda_s$,
while (2) follows directly from the smallness of $s$ and the fact
that $\supp(f) \subset 2B_1\cap\pom$.  Next, let
$z\in \pom$ be a Lebesgue point (with respect to the measure $\hm^X$)
for the function $1_{A_1}$,
so that
$$1_{A_1}(z)\, =\,\lim_{s\to 0}\int_{\pom}  \Lambda_s(z,y) \, 1_{A_1}\,d\sigma(y) \,
\leq \, \liminf_{s\to 0} f_s(z)\,,$$
by the last inequality in \eqref{eq3.5}.
Since $\hm^X$-a.e. $z\in \pom$ is a Lebesgue point, we obtain (3).

To prove (4), we observe that the second inequality is simply a re-statement of
the second inequality in \eqref{eq3.5}, so it suffices to show that
\begin{equation}\label{eq3.12}
\|f_s\|_{BMO(\pom)}\lesssim \|f\|_{BMO(\pom)} \,,\quad {\rm uniformly \, in}\, s\,.
\end{equation}
To this end, we fix a surface ball $\Delta = \Delta(x,r)$, and we consider two cases.

\smallskip

\noindent {\bf Case 1}: $s\geq r$.  In this case, set $c:= \fint_{\Delta(x,2s)} f$,  so that
by ADR, \eqref{eq3.11} and the construction of $\Lambda_s$,
$$\fint_\Delta |f_s - c| \, d\sigma \,\lesssim \, \fint_\Delta \fint_{\Delta(x,2s)} |f-c|\, d\sigma\,
\lesssim \, \|f\|_{BMO(\pom)} \,.$$

\smallskip

\noindent {\bf Case 2}:  $s<r$.  In this case, set $c:= \fint_{2\Delta} f$.
Then by Fubini's Theorem,
$$\fint_\Delta |f_s(z) - c| \, d\sigma(z) \,\lesssim\, \fint_{2\Delta} |f(y)-c|\, \int_{\pom} \Lambda_s(z,y) \,
d\sigma(z) \, d\sigma(y)\,\lesssim \,  \|f\|_{BMO(\pom)} \,,$$
where again we have used ADR, \eqref{eq3.11} and the compact support property
of $\Lambda_s(z,y)$.

Since these bounds are uniform over all $x\in \pom$, and $r\in (0,\diam(\pom))$,
we obtain \eqref{eq3.12}.
\end{proof}

\section{Proof of Theorem \ref{t2}:  $\hm\in$ weak-$A_\infty$ implies $L^p$ and
BMO-solvability}\label{s3}
In this section, we suppose that $\Omega$ is an open set  with ADR boundary $\pom$, and that 
for every ball $B_0=B(x_0,r)$, with $x_0\in \pom$, and $0<r<\diam(\pom)$, and for all 
$Y\in \Omega\setminus 4B_0$,
elliptic-harmonic measure $\hm_L^Y\in$ weak-$A_\infty(\Delta_0)$, where $\Delta_0 := B_0\cap\pom$.
Thus, $\hm_L^Y \ll \sigma$ in $\Delta$, and the Poisson kernel $k^Y:= d\omega_L/d\sigma$ satisfies
 the weak reverse H\"older condition \eqref{eq1.wRH}, for some uniform $q>1$.
In our proof of BMO-solvability (but not for $L^p$ solvability), 
we shall also require, at precisely one point in the argument,
that the Corkscrew condition
(Definition \ref{def1.cork}) is satisfied in $\Omega$.    Even in the absence of the Corkscrew condition,
it may happen that there is a Corkscrew point $X_\Delta$ relative to some particular $\Delta$
(e.g., for every $X\in \Omega$, this is true for the surface ball
$\Delta_X$ as in \eqref{eq3.1}, with $X$ itself serving as a Corkscrew point), and in this case, we have the
following consequence of the weak-$RH_q$ estimate:
\begin{equation}\label{eq1.wRH-2}
\left(\fint_\Delta \left(k^{X_\Delta}\right)^q d\sigma \right)^{1/q} \,\leq\, C\, \sigma(\Delta)^{-1}\,.
\end{equation}
Indeed, one may  cover $\Delta$ 
by a collection of surface balls $\{\Delta'=B'\cap\pom\}$, in such a way that
$X_\Delta \in \Omega\setminus 4B'$, but each $\Delta'$ has radius comparable to that of $\Delta$
(hence $\sigma(\Delta') \approx \sigma(\Delta)$, by the ADR property), depending on the constant
in the Corkscrew condition, and such that the cardinality of the collection $\{\Delta'\}$ is uniformly bounded;
one may then readily derive \eqref{eq1.wRH-2} by applying \eqref{eq1.wRH} 
in each $\Delta'$, and using the crude estimate that $\hm^{X_\Delta}(2\Delta')/\sigma(\Delta')
\leq \sigma(\Delta')^{-1} \approx \sigma(\Delta)^{-1}$.

Our first step is to establish an $L^p$ solvability result.
To this end, we define non-tangential ``cones" and maximal functions, as follows.
First, we fix a collection of standard Whitney cubes covering
$\Omega$, and we denote this collection by $\W$.
Given $x\in \pom$, set
\begin{equation}\label{eq4.cone1}
\W(x):=\{ I\in\W:\, \dist(x,I) \leq 100 \diam(I)\}\,,
\end{equation}
and define
 the (possibly disconnected) non-tangential ``cone" with vertex at $x$ by
\begin{equation}\label{eq4.cone}
\U(x)\,:=\, \cup_{I\in \W(x)} 
\,.\end{equation}
For a continuous $u$ defined on $\Omega$, the non-tangential maximal function of $u$
is defined by
\begin{equation}\label{eq4.nt} 
N_* u(x)\,:=\, \sup_{Y\in \U(x)} |u(Y)| 
\,.\end{equation}
Recall that $\m$ denotes the (non-centered)
Hardy-Littlewood maximal operator on $\pom$.
We have the following.
\begin{proposition}\label{prop4.2}  Suppose that there is a $q>1$, such that \eqref{eq1.wRH}
holds for the Poisson kernel $k^Y$, for
every surface ball $\Delta = B\cap\pom$, 
centered on $\pom$, provided $Y\in \Omega\setminus 4B$.  
Given $g$ continuous with compact support on $\pom$,  
let $u$ be the solution of the Dirichlet problem for
$L$ with data $g$.  Then for $p=q/(q-1)$, and for all $x\in \pom$
\begin{equation}\label{eq4.3}
N_* u(x) \lesssim\left( \m(|g|^p)(x)\right)^{1/p}\,.
\end{equation}
Thus, for all $s>p$, the Dirichlet problem is $L^s$-solvable, i.e., 
\begin{equation}\label{eq4.4}
\|N_*u\|_{L^{s}(\pom)}\, \leq \,C_s\, \|g\|_{L^s(\pom)}\,.
\end{equation}
\end{proposition}
\begin{remark}\label{r4.7a}
As is well known, the weak-$RH_q$ estimate \eqref{eq1.wRH} is self-improving, i.e.,
weak-$RH_q$ implies weak-$RH_{q+\eps}$, for some $\eps>0$, thus, in particular,
one may self-improve
\eqref{eq4.4} to the case $s=p$.  We also remark that our definition of $L^p$-solvability
of the Dirichlet problem entails only a non-tangential maximal function estimate, and does not address
the issue of non-tangential convergence a.e. to the data.  The latter would seem to require
that the Whitney boxes in the definition of $\W(x)$ (see \eqref{eq4.cone1}) 
exist at infinitely many scales, for a.e. $x\in \pom$; e.g.,
the interior Corkscrew condition would be more than enough to guarantee this property.
\end{remark}
\begin{proof}[Proof of Proposition \ref{prop4.2}]  
Splitting the data $g$ into its positive and negative parts, we may suppose without loss of generality that
$g\geq 0$, hence also $u\geq 0$.
Let $x\in\pom$, 
fix $Y\in \U(x)$,  and let $\hat{y}\in \pom$ be a touching point, i.e., 
$|Y-\hat{y}|=\delta(Y)$.  Set 
$$\dsy:= \Delta\big(\hat{y},1000 \delta(Y)\big)\,,\qquad B^*_Y:= B\big(\hat{y},1000 \delta(Y)\big)\,,$$
and note that $x\in \dsy$.
Define a continuous partition of unity 
$\sum_{k\geq 0}\vp_k \equiv 1$ on $\pom$, such that $0\leq \vp_k\leq 1$ for all $k\geq 0$,
with
\begin{equation}\label{eq4.5a}
\supp(\vp_0) \subset 4\dsy\,, \quad 
\supp(\vp_k) \subset R_k:= 2^{k+2}\dsy\setminus 2^{k}\dsy\,,\,\, k\geq 1\,,
\end{equation}
set $g_k:= g\vp_k$, and let $u_k$ be the solution of the Dirichlet problem with data $g_k$.
Thus, $u =\sum_{k\geq 0} u_k$ in $\Omega $.
By construction, $Y$ is a Corkscrew point for $4\dsy$, and
$x\in 4\dsy$, hence
$$u_0(Y) \leq \int_{\pom} g_0 \,  k^Y\, d\sigma\, \lesssim\,
\left(\fint_{4\dsy} g^p\,d\sigma\right)^{1/p} \,\lesssim \, \left( \m\big(g^p\big)(x)\right)^{1/p}\,,$$
where in the next to last step we have used \eqref{eq1.wRH-2}.


Next, we claim that
\begin{equation}\label{eq4.6a} u_k(Y) \lesssim \, 2^{-k\alpha}  \left( \m\big(g^p\big)(x)\right)^{1/p}\,.
\end{equation}
Given this claim, we may sum in $k$ to obtain \eqref{eq4.3}.  Thus, it suffices to verify \eqref{eq4.6a}.
To this end, we set
$$\W_k:= \big\{I\in\W:\, I\,\, {\rm meets}\,\, 2^{k-1}B_Y^*\big\}\,,
$$
and for each $I\in\W_k$, we fix a point $X_I\in I\cap 2^{k-1}B_Y^*$, and we 
define
$$\Delta_I:= \Delta_{X_I}\,,$$
as in \eqref{eq3.1}, with $X=X_I$.   We now choose a collection of balls $\{B_i\}_{1\leq i\leq N}$,
with $N$ depending only on $n$ and ADR, 
and corresponding surface ball $\Delta_i:= B_i\cap\pom$,
such that $R_k\subset \cup_{i=1}^N \Delta_i$, and such that for each $i=1,2,...,N$,
$$ r_{B_i}\approx 2^k r\quad {\rm and}\quad 2^{k-1}B_Y^*\subset 
\ree\setminus 4B_i\,.$$
Then by definition of $R_k$ (see \eqref{eq4.5a}), and the ADR property,
\begin{multline}\label{eq4.7a}
u_k(X_I) \leq\, \int_{R_k}g \, d\hm^{X_I} \,\lesssim\,
(2^k r)^n \left(\fint_{2^{k+2}\dsy} g^p\,d\sigma\right)^{1/p}\left(\sum_{i=1}^N \fint_{\Delta_i}
\left(k^{X_I}\right)^q\, d\sigma \right)^{1/q}  \\[4pt]
\lesssim\, \left(\fint_{2^{k+2}\dsy} g^p\,d\sigma\right)^{1/p}\,
\lesssim \, \left( \m\big(g^p\big)(x)\right)^{1/p}\,,
\end{multline}
where in the next-to-last step we have used the weak-$RH_q$ estimate 
 \eqref{eq1.wRH} in each $\Delta_i$, along with the crude bound $\hm^{X_I}(2\Delta_i) \leq 1$, 
 and the fact that each $\Delta_i$ has radius $r_{\Delta_i}\approx 2^k r$.

Next, by Corollary \ref{cor2.4},
\begin{multline*}
u_k(Y)  \lesssim\, 2^{-k\alpha}\,\, 
\frac1{|2^{k-1}B^*_Y|}\iint_{2^{k-1}B^*_Y\cap\Omega}u_k(Z)\, dZ\\[4pt]
 \lesssim \, 2^{-k\alpha} \frac{1}{(2^k r)^{n+1}}
\sum_{I\in\W_k} \iint_I u_k(Z)\, dZ\,
 \approx \,
2^{-k\alpha} \frac{1}{(2^k r)^{n+1}}
\sum_{I\in\W_k} |I| \, u_k(X_I) \\[4pt]
 \lesssim \, 2^{-k\alpha}\left( \m\big(g^p\big)(x)\right)^{1/p}\,,
\end{multline*}
where in the last two lines we have used Harnack's inequality in the Whitney box $I$,
and  then \eqref{eq4.7a}, and 
the fact that the Whitney boxes in $\W_k$
are non-overlapping and are all contained in a Euclidean ball of radius $\approx 2^k r$.
\end{proof}

With Proposition \ref{prop4.2} in hand, we turn to the proof of BMO-solvability.
Our approach here follows that in \cite{DKP}, which in turn is based on that of \cite{FN}.
We now suppose that the Corkscrew condition holds in $\Omega$, and that $L$ is the Laplacian.  
In this case, by the result of
\cite{HM-IV} (see also \cite{HLMN} and \cite{MT}), 
the weak-$A_\infty$ condition for harmonic measure implies that $\pom$ is uniformly rectifiable,
and thus, by a result of \cite{HMM2}, we have the following
square function/non-tangential maximal function
estimate:  for $u$ harmonic in $\Omega$,
\begin{equation}\label{eq4.10a}
\int_{\pom} \big(\s u\big)^p\, d\sigma \leq C_p \int_{\pom} \big(N_{*} u\big)^p\, d\sigma\,,
\end{equation}
where $C_p$ depends also on $n$, and the UR constants for $\pom$ (and thus on the ADR,
Corkscrew and weak-$A_\infty$ constants), and where 
$$\s u(x):= \left(\iint_{\U(x)} |\nabla u(Y)|^2\, \delta(Y)^{1-n}\, dY\right)^{1/2},$$
and $\U(x)$ and $N_*u$ were defined in \eqref{eq4.cone} and \eqref{eq4.nt}.

Now consider a
ball $B=B(x,r)$, with $x\in \pom$, and $0<r<\diam(\pom)$, and corresponding
surface ball $\Delta = B\cap \pom$.
Let
$f$ be continuous with compact support on $\pom$,  
and set $h:= f- c_\Delta$, where $c_{\Delta}:= \fint_{40\Delta} f$.
We construct a smooth partition of unity $\sum_{k\geq 0} \vp_k \equiv 1$ on $\pom$ as before,
but now with $10\Delta$ in place of $\Delta_Y^*$.   Set
$h_k:= h\vp_k$, and let $u_k$ be the solution to the Dirichlet problem with data $h_k$.
Set
$$  \W_B:= \big\{I\in\W:\, I\,\, {\rm meets}\,\, B\big \}\,,\qquad 
\W_B^j:= \big\{I\in\W_B:\, 
\ell(I) = 2^{-j}\big\}\,,$$
and for each $I\in \W_B$, fix a point $X_I\in I\cap B$.  As above, let $\Delta_I:= \Delta_{X_I}$
be defined as in \eqref{eq3.1}, and note that  by construction,
 \begin{equation*}
 z\in \Delta_I \quad \implies\quad I\in \W(z)\,,
 \end{equation*}
 where $\W(z)$ is defined in \eqref{eq4.cone1}.
Consequently, given $z\in \pom$,
 \begin{equation}\label{eq4.13}
  \sum_{I:\, z\in \Delta_I } \iint_I |\nabla u_0(Y)|^2 \,\delta(Y)^{1-n}\, dY  \,\lesssim \,
 \big( Su(z)\big)^2\,.
 \end{equation} 
 Let us note also that 
 \begin{equation}\label{eq4.14}
 I\in \W_B \,\, \implies\,\, \Delta_I \subset \Delta(x,Cr)=:\Delta^*\,,
 \end{equation}
 for $C$ chosen large enough.
We then have
\begin{equation*}\begin{split}
\iint_{B\cap\Omega}|\nabla u_0(Y)|^2\,\delta(Y)\, dY \,&\lesssim\,
\sum_{I\in\W_B}  \iint_I |\nabla u_0(Y)|^2 \,\delta(Y)\, dY \\
& \approx \,   \sum_{I\in\W_B} \fint_{\Delta_I} 
\iint_I |\nabla u_0(Y)|^2 \,\delta(Y)\, dY\, d\sigma  \\
& \lesssim \, \int_{\Delta^*} \big( \s u_0(z)\big)^2 \, d\sigma(z) \\ & \lesssim\,
\sigma(\Delta)^{(p-2)/p}\,\left(\int_{\Delta^*} \big( \s u_0(z)\big)^p \, d\sigma(z)\right)^{2/p}\,,
\end{split}
\end{equation*}
where in the last two steps we have used the ADR property and  \eqref{eq4.13},
and then ADR again.  Therefore, by \eqref{eq4.10a}, and then  Remark \ref{r4.7a},
and the definition of $u_0$,
$$\frac1{\sigma(\Delta)}\iint_{B\cap\Omega}|\nabla u_0(Y)|^2\,\delta(Y)\, dY \,\lesssim\,
\sigma(\Delta)^{-2/p}  \left(\int_{40\Delta}|f-c_\Delta|^p\right)^{2/p} \lesssim \|f\|^2_{BMO(\pom)}\,.$$

For $k\geq 1$, we set $g_k:= |h_k| = |f-c_\Delta|\vp_k$, and let $v_k$ be the solution of the Dirichlet
problem with data $g_k$.  Thus, $|u_k| \leq v_k$.    For $k\geq 0$, set 
$$\widetilde{B}:= 40 B =B(x,40r)\,,\quad B_k:= 2^{k}\, \widetilde{B}\,,\quad \Delta_k := B_k\cap\pom\,,$$ 
and let $\Delta_k^*$ be a sufficiently large concentric fattening of $\Delta_k$.  Given $I\in\W$,
define $I^*=((1+\tau)I$, with $\tau$ chosen small enough that 
$\dist(I^*,\pom) \approx \dist(I,\pom)\approx \diam(I)$.
Then for $Y\in I^*$, with $I\in \W_B^j$, by Corollary \ref{cor2.4},
\begin{multline*}
v_k(Y)\,\lesssim \, \left(\frac{\ell(I)}{2^kr}\right)^{\alpha} \frac1{|B_{k-1}|}
\iint_{B_{k-1}\cap\Omega}v_k\lesssim \, 
\big(2^j2^k r\big)^{-\alpha} \fint_{\Delta_k^*}N_*v_k\, d\sigma \\
\lesssim \,\big(2^j2^k r\big)^{-\alpha} \left(\fint_{\Delta_k^*}\big(N_*v_k\big)^p\, d\sigma\right)^{1/p}
\lesssim \, \big(2^j2^k r\big)^{-\alpha}  \left(\fint_{\Delta_{k+2}}|f-c_\Delta|^p\, d\sigma\right)^{1/p}\\
\lesssim  \, k \,\big(2^j2^k r\big)^{-\alpha}\, \|f\|_{BMO(\pom)}\,,
\end{multline*}
where in the last two steps we have used Remark \ref{r4.7a}, and a well known telescoping argument.
Consequently, setting $\Delta^*=\Delta(x,Cr)$ as in \eqref{eq4.14},
\begin{equation*}\begin{split}
\iint_{B\cap\Omega}|\nabla u_k(Y)|^2\,\delta(Y)\, dY \,&\lesssim\,
 \sum_{I\in\W_B} \ell(I) \iint_I |\nabla u_k(Y)|^2 \, dY \\
&\lesssim \,   \sum_{I\in\W_B} \ell(I)^{-1}
\iint_{I^*} | u_k(Y)|^2 \, dY\, d\sigma  \\ 
&\lesssim\, k^2\,2^{-2k\alpha}\, \|f\|^2_{BMO(\pom)}
\sum_{j: \, 2^{-j}\lesssim r} \,\big(2^j r\big)^{-2\alpha} \sum_{I\in\W_B^j} 
\sigma(\Delta_I)\\
&\lesssim\, k^2\,2^{-2k\alpha}\, \|f\|^2_{BMO(\pom)}  \, \sigma(\Delta^*)\,,
\end{split}
\end{equation*}
since for each fixed $j$, the surface balls $\Delta_I$ with $I\in \W_B^j$ have bounded overlaps,
and are all contained in $\Delta^*$.  Dividing by $\sigma(\Delta^*)$ and using ADR,
we may then sum in $k$ to obtain \eqref{eq1.2}, thus concluding the proof of Theorem \ref{t2}.

\noindent {\bf Acknowledgements}.  We are grateful to Simon Bortz for a suggestion
which has simplified one of our arguments in Section \ref{s3}.

\end{document}